\documentclass[reqno]{amsart}

\usepackage [mathscr]{eucal}
\usepackage{amsmath,amsthm,amssymb,amscd}
\usepackage{amsrefs}
\usepackage{epsf}
\usepackage{amsfonts}
\usepackage{dsfont}
\usepackage{latexsym}
\usepackage{mathrsfs}
\usepackage{layout}
\usepackage{bm}
\usepackage{verbatim}

\newtheorem{theorem}{Theorem}[section]
\newtheorem{lemma}[theorem]{Lemma}
\newtheorem{proposition}[theorem]{Proposition}
\newtheorem{corollary}[theorem]{Corollary}

\newtheorem{definition}[theorem]{Definition}

\numberwithin{equation}{section}

\def\Xint#1{\mathchoice
{\XXint\displaystyle\textstyle{#1}}%
{\XXint\textstyle\scriptstyle{#1}}%
{\XXint\scriptstyle\scriptscriptstyle{#1}}%
{\XXint\scriptscriptstyle\scriptscriptstyle{#1}}%
\!\int}
\def\XXint#1#2#3{{\setbox0=\hbox{$#1{#2#3}{\int}$ }
\vcenter{\hbox{$#2#3$ }}\kern-.6\wd0}}

\def\dint{\Xint-}

\DeclareMathOperator *{\essosc}{ess\ osc}
\DeclareMathOperator *{\osc}{osc}
\DeclareMathOperator *{\esssup}{ess\ sup}
\DeclareMathOperator *{\essinf}{ess\ inf}
\DeclareMathOperator *{\di}{div} 
\DeclareMathOperator *{\meas}{meas}
\DeclareMathOperator *{\dist}{dist}
\DeclareMathOperator *{\data}{data}
\DeclareMathOperator *{\diam}{diam}
\DeclareMathOperator *{\loc}{loc}
\DeclareMathOperator *{\Lip}{Lip}
\DeclareMathOperator *{\Tr}{Tr}
\DeclareMathOperator *{\BMO}{BMO}
\DeclareMathOperator *{\Proj}{Proj}
\DeclareMathOperator *{\BBR}{\mathbb{R}}
\DeclareMathOperator *{\BBC}{\mathbb{C}}

\newcommand{\cJ}{{\mathcal J}}

\begin{document}
\title[Elliptic PDEs with complex coefficients]{
Perturbation theory for solutions to second order elliptic operators with complex coefficients and the $L^p$ Dirichlet problem}

\author{Martin Dindo\v{s}}
\address{School of Mathematics, \\
         The University of Edinburgh and Maxwell Institute of Mathematical Sciences, UK}
\email{M.Dindos@ed.ac.uk}

\author{Jill Pipher}
\address{Department of Mathematics, \\ 
	Brown University, USA}
\email{jill\_pipher@brown.edu}

\begin{abstract}
We establish a Dahlberg-type perturbation theorem for second order divergence form elliptic  operators with complex coefficients. In \cite{DPcplx}, we showed the following result: If ${\mathcal L}_0=\mbox{div} A^0(x)\nabla+B^0(x)\cdot\nabla$ is a $p$-elliptic operator
satisfying the assumptions of Theorem \ref{S3:T1} then the $L^p$ Dirichlet problem for the operator ${\mathcal L}_0$ is solvable in the upper half-space ${\mathbb R}^n_+$.

In this paper we prove that the $L^p$ solvability is stable under small perturbations of ${\mathcal L}_0$. That is if ${\mathcal L}_1$
is another divergence form elliptic  operator with complex coefficients and the coefficients of the
operators ${\mathcal L}_0$ and ${\mathcal L}_1$ are sufficiently close in the sense of Carleson
measures, then the $L^p$ Dirichlet problem for the operator
${\mathcal L}_1$ is solvable for the same value of $p$. 

As a corollary we obtain a new result on $L^p$ solvability of the Dirichlet problem for operators of the form ${\mathcal L}=\mbox{div} A(x)\nabla+B(x)\cdot\nabla$ where the matrix $A$ satisfies weaker Carleson condition
than in \cite{DPcplx}; in particular the coefficients of $A$ need no longer be differentiable and instead satisfy a Carleson condition that controls the oscillation of the matrix $A$ over Whitney boxes. 
This result in the real case has been established in \cite{DPP}.
\end{abstract}

\maketitle

\section{Introduction}\label{S:Intro}

It is natural to ask when behavior of solutions to, or other properties of, certain partial differential equations are preserved under small perturbations of the coefficients. In the case of second order elliptic and parabolic operators, this is a well-studied question, at least when the coefficients are real valued. In 1986, Dahlberg \cite{Da} studied this question in the context of the solvability of the $L^p$ Dirichlet problem, and gave a criterion for smallness of the perturbation in terms of Carleson measures. We refer to this, and a variety of similar conditions that have arisen over the years, as ``Dahlberg-type" perturbation criteria. The results of \cite{Da} were sharpened in \cite{FKP}, to include consequences of both smallness and simply finiteness of the Carleson measure condition on the perturbation. Subsequent work enlarged the class of domains in which these results
hold, as well as the class of operators. In \cite{Esc}, Escauriaza showed that the small Carleson condition preserved a refined property of the density of elliptic measure, namely that the logarithm of the density belongs to the Sarason space of Vanishing Mean Oscillation (VMO). Dahlberg's theorem and Escauriaza's theorem were both extended to chord-arc domains in \cite{MPT1} and \cite{MPT2}.  The theory has also been extended to parabolic operators; see \cite{RN} for non-cylindrical domains.
For equations in non-divergence form, Rios (\cite{Ri}) showed that the $A_\infty$ property of elliptic measure is preserved under finiteness of the Carleson measure condition, and Dindo\v{s}-Wall gave the sharp result assuming the small Carleson condition.
All of this was done, however, in the case of real coefficients. Perturbation theory in the complex coefficient case is far less developed; we will mention a few examples momentarily.

In the paper \cite{DPcplx} we established a new theory of interior regularity for solutions to complex coefficient second order divergence form operators, which can be viewed as a weaker substitute for the De Giorgi-Nash-Moser regularity theory for real elliptic PDEs. 
Specifically, we considered operators of the form
$\mathcal{L}=\mbox{div} A(x)\nabla +B(x)\cdot\nabla$, with certain algebraic conditions on the matrix $A$ (called $p$-ellipticity) and a natural minimal scaling condition on $B$. No additional smoothness of the coefficients is assumed.
When the coefficients of $A$ are real, or when $p=2$, the $p$-ellipticity condition is just the usual uniform ellipticity condition.

We then applied this regularity theory to the question of solvability of the $L^p$ Dirichlet problem for operators with complex coefficients. We always assume that the matrices are in canonical form, as defined below, and \cite{DPcplx} contains a discussion of how to put an operator with lower terms in canonical form.
In particular, we established the following result.

\begin{theorem}\label{S3:T1}  Let $1<p<\infty$, and let $\Omega$ be the upper half-space ${\mathbb R}^n_+=\{(x_0,x'):\,x_0>0\mbox{ and } x'\in{\mathbb R}^{n-1}\}$. Consider the operator 
$$ \mathcal Lu = \partial_{i}\left(A_{ij}(x)\partial_{j}u\right) 
+B_{i}(x)\partial_{i}u$$
and assume that the matrix $A$ in canonical form is $p$-elliptic with constants $\lambda_p,\Lambda$. Canonical forms mean that $A_{00}=1$ and $\mathscr{I}m\,A_{0j}=0$ for all $1\leq j \leq n-1$.
Assume that
\begin{equation}\label{Car_hatAA}
d{\mu}(x)=\sup_{B_{\delta(x)/2}(x)}\left[|\nabla{A}|^{2} + |B|^{2} \right]\delta(x)\,dx
\end{equation}
is a Carleson measure in $\Omega$. Let us also denote 
\begin{equation}\label{Car_hatAAB}
d{\mu'}(x)=\sup_{B_{\delta(x)/2}(x)}\left[\textstyle\sum_j\left|\partial_0 A_{0j}\right|^{2}+\left|\textstyle\sum_j\partial_j A_{0j}\right|^{2} + |B|^{2} \right]\delta(x)\,dx.
\end{equation}

Then there exist $K=K(\lambda_p,\Lambda,\|\mu\|_{\mathcal C},n,p)>0$ and $C(\lambda_p, \Lambda ,\|\mu\|_{\mathcal C}, n,p)>0$ such that if
\begin{equation}\label{Small-Cond3}
\|\mu'\|_{\mathcal C} < K
\end{equation}
then the $L^p$-Dirichlet problem  
 
 \begin{equation}\label{E:D2new2}
\begin{cases}
\,\,{\mathcal L}u=0 
& \text{in } \Omega,
\\[4pt]
\quad u=f & \text{ for $\sigma$-a.e. }\,x\in\partial\Omega, 
\\[4pt]
\tilde{N}_{p,a}(u) \in L^{p}(\partial \Omega), &
\end{cases}
\end{equation}
is solvable and the estimate
\begin{equation}\label{Main-Est}
\|\tilde{N}_{p,a} (u)\|_{L^{p}(\partial \Omega)}\leq C\|f\|_{L^{p}(\partial \Omega;{\BBC})}
\end{equation}
holds for all energy solutions $u$ with datum $f$.
\end{theorem}

We now aim to replace the criteria (\eqref{Car_hatAA} and \eqref{Car_hatAAB}) under which solvability can be deduced by a weaker condition, which previously was established in \cite{DPP} only for real valued elliptic operators. 
The Carleson conditions \eqref{Car_hatAA} and \eqref{Car_hatAAB} in the theorem above require the matrix $A$ to be differentiable while in \cite{DPP}, 
it was shown that a condition on the oscillation of the coefficients of $A$ suffices. Namely, it was proven that in the real variable case 
Theorem \ref{S3:T1} holds when 
\begin{equation}\label{oscT1}
d\mu =  \left( \delta(x)^{-1}\left(\mbox{osc}_{B_{\delta(x)/2}(x)}A\right)^2 + \sup_{B_{\delta(x)/2}(x)} |B|^2 \delta(x)\right) dx
\end{equation}
is a small Carleson measure. Here we define
$$\mbox{osc}_{K}A=\sup_{i,j}\left\{\left|A_{ij}(x)-A_{ij}(y)\right|:\,\mbox{for }x,y\in K|\right\}.$$

The proof of solvability in \cite{DPP} under this weaker assumption on the coefficients  used a perturbation result of Dahlberg type, which states that if ${\mathcal L}_0$ and ${\mathcal L}_1$ are two divergence form elliptic operators whose
coefficients are close in the sense that 
 \begin{equation}\label{eqdm1}
 dm(x)=\sup_{B_{\delta(x)/2}(x)}\left[|A^0 - A^1|^{2} \delta^{-1}(x)+ |B^0 - B^1|^{2} \delta(x)\right]\,dx
\end{equation}
is a small  Carleson measure, then the $L^p$ solvability of ${\mathcal L}_0$ implies the $L^p$ solvability of ${\mathcal L}_1$.

An analogous result was not known in the complex coefficient case. In fact, all known proofs of the Dahlberg's perturbation theorem rely heavily on properties
of solutions that only hold in the real variable case, such as the maximum principle or the $A_\infty$ property of the elliptic measure. 

In general the literature on solvability of boundary value problems for complex coefficient operators in $\mathbb R^n$ has been fairly limited, except when the matrix $A$ is of block form. When the matrix $A$ in $\mathcal L = \mbox{div} A(x)\nabla$ has block form, there are numerous results on on $L^p$-solvability of the Dirichlet, regularity and Neumann problems, starting with the solution of the Kato problem, where the coefficients of the block matrix are also assumed to be independent of the transverse variable. This assumption on the transverse variable is usually referred in literature as \lq\lq $t$-independent" - in our notation it is the $x_0$ variable - referring to the situation when the domain is the upper half space.
See \cite{AHLMT} and \cite{HM} and the references therein. 
For matrices that are not of block form, there are solvability results in a few special cases, under the assumption that the solutions satisfy De Giorgi - Nash - Moser estimates. This latter assumption is not generally verifiable or, as far as we know, linked to any specific structural assumption on the matrix. For examples of results obtained under this assumption, see \cite{AAAHK} and \cite{HKMPreg}; the latter paper is also concerned with operators that are $t$-independent.

Finally, there are perturbation results in a variety of special cases, such as \cite{AAM} and \cite{AAH}; the first paper shows that solvability in $L^2$ implies solvability in $L^p$ for $p$ near $2$, and the second paper has $L^2$-solvability results for small $L^\infty$ perturbations of real elliptic operators when the complex matrix is $t$-independent.

In this paper we show that Dahlberg's perturbation theory applies to the class of complex coefficient elliptic operators ${\mathcal L}_0$ satisfying the assumptions of Theorem \ref{S3:T1} under only the structural, or algebraic, assumption of $p$-ellipticity.
We do not assume $t$-independence, nor do we assume that solutions satisfy De Giorgi - Nash - Moser estimates.
The perturbation criteria is the same as that for real coefficient operators, although here it should be observed that the ``smallness" of the Carleson measure will also be a function of the $p$ in $p$-ellipticity.

\begin{theorem}\label{MainPerturb}
Let $\mathcal{L}_0$ satisfy the conditions of Theorem \ref{S3:T1}. Let $\mathcal{L}_1=\mbox{\rm div} A^1(x)\nabla +B^1(x)\cdot\nabla$ be a perturbation of
 $\mathcal{L}_0$ in the following sense:
 
 \begin{equation}\label{eqdm}
 dm(x)=\sup_{B_{\delta(x)/2}(x)}\left[|A^0 - A^1|^{2} \delta^{-1}(x)+ |B^0 - B^1|^{2} \delta(x)\right]\,dx
\end{equation}
is a Carleson measure in $\Omega$.

Then there exist $K=K(\lambda_p,\Lambda,\|\mu\|_{\mathcal C},n,p)>0$ and $C(\lambda_p, \Lambda ,\|\mu\|_{\mathcal C}, n,p)>0$ such that if
\begin{equation}\label{Small-Cond}
\|m\|_{\mathcal C} < K
\end{equation}
then the matrix $A^1$ is p-elliptic and the $L^p$-Dirichlet problem is also solvable for the operator $\mathcal{L}_1$.
That is,

 \begin{equation}\label{E:D2new}
\begin{cases}
\,\,{\mathcal L_1}u=0 
& \text{in } \Omega,
\\[4pt]
\quad u=f & \text{ for $\sigma$-a.e. }\,x\in\partial\Omega, 
\\[4pt]
\tilde{N}_{p,a}(u) \in L^{p}(\partial \Omega), &
\end{cases}
\end{equation}
is solvable and the estimate
\begin{equation}\label{Main-Est2}
\|\tilde{N}_{p,a} (u)\|_{L^{p}(\partial \Omega)}\leq C\|f\|_{L^{p}(\partial \Omega;{\BBC})}
\end{equation}
holds for all energy solutions $u$ of $\mathcal{L}_1u=0$ with datum $f$.
\end{theorem}

This perturbation result is sufficient to establish a new criterion for solvability of the $L^p$ Dirichlet problem for operators with complex coefficients, analogous to the one in \cite{DPP} in the real case.

\begin{theorem}\label{Tmain}
 Let $1<p<\infty$, and let $\Omega$ be the upper half-space ${\mathbb R}^n_+=\{(x_0,x'):\,x_0>0\mbox{ and } x'\in{\mathbb R}^{n-1}\}$. Consider the operator 
$$ \mathcal Lu = \partial_{i}\left(A_{ij}(x)\partial_{j}u\right) 
+B_{i}(x)\partial_{i}u$$
and assume that the matrix $A$ is $p$-elliptic with constants $\lambda_p,\Lambda$, $A_{00}=1$ and $\mathscr{I}m\,A_{0j}=0$ for all $1\leq j \leq n-1$.
Assume that

\begin{equation}\label{oscT}
d\mu =  \left( \delta(x)^{-1}\left(\mbox{osc}_{B_{\delta(x)/2}(x)}A\right)^2 + \sup_{B_{\delta(x)/2}(x)} |B|^2 \delta(x)\right) dx
\end{equation}
is a Carleson measure in $\Omega$. Let
\begin{equation}\label{oscT2}
d\mu' =  \left( \delta(x)^{-1}\sum_{j=0}^{n-1}\left(\mbox{osc}_{B_{\delta(x)/2}(x)}A_{0j}\right)^2 + \sup_{B_{\delta(x)/2}(x)} |B|^2 \delta(x)\right) dx.
\end{equation}

Then there exist $K=K(\lambda_p,\Lambda,\|\mu\|_{\mathcal C},n,p)>0$ and $C(\lambda_p, \Lambda ,\|\mu\|_{\mathcal C}, n,p)>0$ such that if
\begin{equation}\label{Small-Cond2}
\|\mu'\|_{\mathcal C} < K
\end{equation}
then the $L^p$-Dirichlet problem  
 
 \begin{equation}\label{E:D2}
\begin{cases}
\,\,{\mathcal L}u=0 
& \text{in } \Omega,
\\[4pt]
\quad u=f & \text{ for $\sigma$-a.e. }\,x\in\partial\Omega, 
\\[4pt]
\tilde{N}_{p,a}(u) \in L^{p}(\partial \Omega), &
\end{cases}
\end{equation}
is solvable and the estimate
\begin{equation}\label{Main-Est3}
\|\tilde{N}_{p,a} (u)\|_{L^{p}(\partial \Omega)}\leq C\|f\|_{L^{p}(\partial \Omega;{\BBC})}
\end{equation}
holds for all energy solutions $u$ with datum $f$.
\end{theorem}

\begin{corollary}\label{block} Suppose the operator $\mathcal L$ on $\mathbb R^n_+$ has the form
$${\mathcal L}u=\partial^2_0 u +\sum_{i,j=1}^{n-1}\partial_i(A_{ij}\partial_j u)$$
where the matrix $A$ has coefficients satisfying the Carleson condition \eqref{oscT}.

Then for all $1<p<\infty$ for which $A$ is $p$-elliptic, the $L^p$-Dirichlet problem  \eqref{E:D} is solvable for $\mathcal L$
and the estimate
\begin{equation}\label{Main-Est2z}
\|\tilde{N}_{p,a} u\|_{L^{p}(\partial \Omega)}\leq C\|f\|_{L^{p}(\partial \Omega;{\BBC})}
\end{equation}
holds for all energy solutions $u$ with datum $f$. 
\end{corollary}

In the statements of the two theorems above, we have used some notation that will be precisely defined in the following section. In section 3 we prove bounds for the square function in terms of the boundary data and the nontangential maximal function. Section 4 contains the converse estimates. Finally, in section 5 we give proofs of Theorems \ref{MainPerturb} and \ref{Tmain}.

\section{Basic notions and definitions}

\subsection{$p$-ellipticity}
The concept of $p$-ellipticity
was introduced in \cite{CM}, where the authors investigated the $L^p$-dissipativity of second order divergence complex coefficient operators. Later, Carbonaro and Dragi\v{c}evi\'c \cite{CD} gave an equivalent definition and coined the term ``$p$-ellipticity".  It is this definition that was
most useful for the results of  \cite{DPcplx}. 
To introduce this, we define, for $p>1$, the ${\mathbb R}$-linear map $\cJ_p:{\mathbb C}^n\to {\mathbb C}^n$ by
$$\cJ_p(\alpha+i\beta)=\frac{\alpha}{p}+i\frac{\beta}{p'}$$
where $p'=p/(p-1)$ and $\alpha,\beta\in{\mathbb R}^n$.

\begin{definition}\label{pellipticity} Let $\Omega\subset{\mathbb R}^n$. Let $A:\Omega\to M_n(\mathbb C)$, where $M_n(\mathbb C)$ is the space of $n\times n$ complex valued matrices. We say that $A$ is $p$-elliptic if for a.e. $x\in\Omega$
\begin{equation}\label{pEll}
\mathscr{R}e\,\langle A(x)\xi,\cJ_p\xi\rangle \ge \lambda_p|\xi|^2,\qquad\forall \xi\in{\mathbb C}^n
\end{equation}
for some $\lambda_p>0$ and there exists $\Lambda>0$ such that 
\begin{equation}
|\langle A(x)\xi,\eta \rangle| \le \Lambda |\xi| |\eta|, \qquad\forall \xi, \,\eta\in{\mathbb C}^n.
\end{equation}
\end{definition}

It is now easy to observe that the notion of $2$-ellipticity coincides with the usual ellipticity condition for complex matrices.
As shown in \cite{CD} if $A$ is elliptic, then there exists $\mu(A)>0$ such that $A$ is $p$-elliptic if and only if
$\left|1-\frac2p\right|<\mu(A).$
Also $\mu(A)=\infty$ if and only if $A$ is real valued.

\subsection{Nontangential maximal and square functions}
\label{SS:NTS}

On a domain of the form 
\begin{equation}\label{Omega-111}
\Omega=\{(x_0,x')\in\BBR\times{\BBR}^{n-1}:\, x_0>\phi(x')\},
\end{equation}
where $\phi:\BBR^{n-1}\to\BBR$ is a Lipschitz function with Lipschitz constant given by 
$L:=\|\nabla\phi\|_{L^\infty(\BBR^{n-1})}$, define for each point $x=(x_0,x')\in\Omega$
\begin{equation}\label{PTFCC}
\delta(x):=x_0-\phi(x')\approx\mbox{dist}(x,\partial\Omega).
\end{equation}
In other words, $\delta(x)$ is comparable to the distance of the point $x$ from the boundary of $\Omega$.

\begin{definition}\label{DEF-1}
A cone of aperture $a>0$ is a non-tangential approach region to the point $Q=(x_0,x') \in \partial\Omega$ defined as
\begin{equation}\label{TFC-6}
\Gamma_{a}(Q)=\{(y_0,y')\in\Omega:\,a|x_0-y_0|>|x'-y'|\}.
\end{equation}
\end{definition}

We require $1/a>L$, otherwise the aperture of the cone is too large and might not lie inside $\Omega$.  When $\Omega=\BBR^n_+$ all parameters $a>0$ may be considered.
Sometimes it is necessary to truncate $\Gamma(Q)$ at height $h$, in which case we write
\begin{equation}\label{TRe3}
\Gamma_{a}^{h}(Q):=\Gamma_{a}(Q)\cap\{x\in\Omega:\,\delta(x)\leq h\}.
\end{equation}

\begin{equation}\label{SSS-1}
\|S_{a}(u)\|^{2}_{L^{2}(\partial\Omega)}\approx\int_{\Omega}|\nabla u(x)|^{2}\delta(x)\,dx.
\end{equation}

In [DPP], a ``$p$-adapted" square function was introduced. The usual square function is the $p$-adapted square function when $p=2$. In the following definition, when $p<2$ we use the convention that the expression $|\nabla u(x)|^{2} |u(x)|^{p-2}$ is zero whenever
$\nabla u(x)$ vanishes.

\begin{definition}\label{D:Sp}
For $\Omega \subset \mathbb{R}^{n}$, the $p$-adapted square function of 
$u\in W^{1,2}_{loc}(\Omega; {\BBC})$ at $Q\in\partial\Omega$ relative 
to the cone $\Gamma_{a}(Q)$ is defined by
\begin{equation}\label{yrddp}
S_{p,a}(u)(Q):=\left(\int_{\Gamma_{a}(Q)}|\nabla u(x)|^{2} |u(x)|^{p-2}\delta(x)^{2-n}\,dx\right)^{1/2}
\end{equation}
and, for each $h>0$, its truncated version is given by 
\begin{equation}\label{yrddp.2}
S_{p,a}^{h}(u)(Q):=\left(\int_{\Gamma_{a}^{h}(Q)}|\nabla u(x)|^{2}|u(x)|^{p-2}\delta(x)^{2-n}\,dx\right)^{1/2}.
\end{equation}
\end{definition}

It is not immediately clear that the integrals appearing in \eqref{yrddp} are well-defined. However, in
\cite{DPcplx}, it was shown that the expressions of the form $|\nabla u(x)|^{2} |u(x)|^{p-2}$, when $u$ is a solution of $\mathcal Lu=0$, are locally integrable and hence the definition of $S_p(u)$ makes sense for such $p$ whenever $p$-ellipticity holds.

A simple application of Fubini's theorem gives 
\begin{equation}\label{SSS-2}
\|S_{p,a}(u)\|^{p}_{L^{p}(\partial\Omega)}\approx\int_{\Omega}|\nabla u(x)|^{2}|u(x)|^{p-2}\delta(x)\,dx.
\end{equation}

\begin{definition}\label{D:NT} 
For $\Omega\subset\mathbb{R}^{n}$ as above, and for 
a continuous $u: \Omega \rightarrow \mathbb C$, the nontangential maximal function ($h$-truncated nontangential maximal function) of $u$
 at $Q\in\partial\Omega$ relative to the cone $\Gamma_{a}(Q)$,
is defined by
\begin{equation}\label{SSS-2a}
N_{a}(u)(Q):=\sup_{x\in\Gamma_{a}(Q)}|u(x)|\,\,\text{ and }\,\,
N^h_{a}(u)(Q):=\sup_{x\in\Gamma^h_{a}(Q)}|u(x)|.
\end{equation}
Moreover, we shall also consider a related version  of the above nontangential maximal function.
This is denoted by $\tilde{N}_{p,a}$ and is defined using $L^p$ averages over balls in the domain $\Omega$. 
Specifically, given $u\in L^p_{loc}(\Omega;{\BBC})$ we set
\begin{equation}\label{SSS-3}
\tilde{N}_{p,a}(u)(Q):=\sup_{x\in\Gamma_{a}(Q)}w(x)\,\,\text{ and }\,\,
\tilde{N}_{p,a}^{h}(u)(Q):=\sup_{x\in\Gamma_{a}^{h}(Q)}w(x)
\end{equation}
for each $Q\in\partial\Omega$ and $h>0$ where, at each $x\in\Omega$, 
\begin{equation}\label{w}
w(x):=\left(\dint_{B_{\delta(x)/2}(x)}|u(z)|^{p}\,dz\right)^{1/p}.
\end{equation}
\end{definition}

Above and elsewhere, a barred integral indicates an averaging operation. Observe that, given $u\in L^p_{loc}(\Omega;{\BBC})$, the function $w$ 
associated with $u$ as in \eqref{w} is continuous and $\tilde{N}_{p,a}(u)=N_a(w)$ 
everywhere on $\partial\Omega$.

The $L^2$-averaged nontangential maximal function was introduced in \cite{KP2} in connection with
the Neuman and regularity problem value problems. In the context of $p$-ellipticity, Proposition 3.5 of \cite{DPcplx} shows that there is no difference between 
$L^2$ averages and $L^p$ averages and that $\tilde{N}_{p,a}(u)$ and $\tilde{N}_{2,a'}(u)$ are comparable in $L^r$ norms for all $r>0$ and all allowable apertures $a,a'$.

\subsection{Carleson measures}
\label{SS:Car} 

We begin by recalling the definition of a Carleson measure in a domain $\Omega$ as in \eqref{Omega-111}. 
For $P\in{\BBR}^n$, define the ball centered at $P$ with the radius $r>0$ as
\begin{equation}\label{Ball-1}
B_{r}(P):=\{x\in{\BBR}^n:\,|x-P|<r\}.
\end{equation}
Next, given $Q \in \partial\Omega$, by $\Delta=\Delta_{r}(Q)$ we denote the surface ball  
$\partial\Omega\cap B_{r}(Q)$. The Carleson region $T(\Delta_r)$ is then defined by
\begin{equation}\label{tent-1}
T(\Delta_{r}):=\Omega\cap B_{r}(Q).
\end{equation}

\begin{definition}\label{Carleson}
A Borel measure $\mu$ in $\Omega$ is said to be Carleson if there exists a constant $C\in(0,\infty)$ 
such that for all $Q\in\partial\Omega$ and $r>0$
\begin{equation}\label{CMC-1}
\mu\left(T(\Delta_{r})\right)\leq C\sigma(\Delta_{r}),
\end{equation}
where $\sigma$ is the surface measure on $\partial\Omega$. 
The best possible constant $C$ in the above estimate is called the Carleson norm 
and is denoted by $\|\mu\|_{\mathcal C}$.
\end{definition}

In all that follows we now assume that the coefficients of the matrix $A$ and $B$ of the elliptic operator $\mathcal{L}=\mbox{div} A(x)\nabla +B(x)\cdot\nabla$  
satisfies the following natural conditions.  
First, we assume that the entries $A_{ij}$ of $A$ are in ${\rm Lip}_{loc}(\Omega)$ and the entries of $B$ are $L^\infty_{loc}(\Omega)$. 
Second, we assume that
\begin{equation}\label{CarA}
d\mu(x)=\sup_{B_{\delta(x)/2}(x)}[|\nabla A|^2+|B|^2]\delta(x) \,dx
\end{equation}
is a Carleson measure in $\Omega$. Sometimes, and for certain coefficients of $A$, we will
assume that their Carleson norm $\|\mu\|_{\mathcal{C}}$ is sufficiently small. 
The fact that $\mu$ is a Carleson allows one to relate integrals in $\Omega$ with respect to $\mu$ to boundary integrals involving 
the nontangential maximal function.
We will often use the following result for our averaged nontangential maximal function, which is Theorem 3.7 of \cite{DPcplx}.

\begin{theorem}\label{T:Car}
Suppose that $d\nu=f\,dx$ and $d\mu(x)=\left[\sup_{B_{\delta(x)/2}(x) }|f|\right]dx$. Assume that
$\mu$ is a Carleson measure. Then there exists a finite 
constant $C=C(L,a)>0$ such that for every $u\in L^{p}_{loc}(\Omega;{\BBC})$ one has
\begin{equation}\label{Ca-222}
\int_{\Omega}|u(x)|^p\,d\nu(x)\leq C\|\mu\|_{\mathcal{C}} 
\int_{\partial\Omega}\left(\tilde{N}_{p,a}(u)\right)^p\,d\sigma.
\end{equation}
Furthermore, consider $\Omega={\mathbb R}^n_+$ where  $\mu$ and $\nu$ are measures as above supported in $\Omega$ and $\delta(x_0,x')=x_0$. Let
 $h:{\mathbb R}^{n-1}\to {\mathbb R}^+$ be a Lipschitz function with Lipschitz norm $L$
and 
$$\Omega_h=\{(x_0,x'):x_0>h(x')\}.$$
Then for any $\Delta\subset {\mathbb R}^{n-1}$ with $\sup_{\Delta} h\le \mbox{diam}(\Delta)/2$ we have
\begin{equation}\label{Ca-222-x}
\int_{\Omega_h\cap T(\Delta)}|u(x)|^p\,d\nu(x)\leq C\|\mu\|_{\mathcal{C}} 
\int_{\partial\Omega_h\cap T(\Delta)}\left(\tilde{N}_{p,a,h}(u)\right)^p\,d\sigma.
\end{equation}
Here for a point $Q=(h(x'),x')\in\partial\Omega_h$ we define
\begin{equation}
\tilde{N}_{p,a,h}(u)(Q) = \sup_{\Gamma_a(Q)}w,\label{eq-Nh}
\end{equation}
where
\begin{equation}\label{TFC-6x}
\Gamma_{a}(Q)=\Gamma_{a}((h(x'),x'))=\{y=(y_0,y')\in\Omega:\,a|h(x')-y_0|>|x'-y'|\}
\end{equation}
and the $L^p$ averages $w$ are defined by \eqref{w} where the distance $\delta$ is taken with respect to the domain $\Omega={\mathbb R}^n_+$.
\end{theorem}

\subsection{Pullback Transformation}
\label{SS:PT}
The Carleson measure conditions on the coefficients of  $\mathcal L$ given in \eqref{CarA}
are compatible with a useful change of variables described in this subsection. 

For a domain $\Omega$ as in \eqref{Omega-111}, consider the mapping 
$\rho:\mathbb{R}^{n}_{+}\to\Omega$ appearing in works of Dahlberg, Ne\v{c}as, 
Kenig-Stein and others, defined by
\begin{equation}\label{E:rho}
\rho(x_0, x'):=\big(x_0+P_{\gamma x_0}\ast\phi(x'),x'\big),
\qquad\forall\,(x_0,x')\in\mathbb{R}^{n}_{+},
\end{equation}
for some positive constant $\gamma$. Here $P$ is a nonnegative function 
$P\in C_{0}^{\infty}(\mathbb{R}^{n-1})$ and, for each $\lambda>0$,  
\begin{equation}\label{PPP-1a}
P_{\lambda}(x'):=\lambda^{-n+1}P(x'/\lambda),\qquad\forall\,x'\in{\mathbb{R}}^{n-1}.
\end{equation}
Finally, $P_{\lambda}\ast\phi(x')$ is the convolution
\begin{equation}\label{PPP-lambda}
P_{\lambda}\ast\phi(x'):=\int_{\mathbb{R}^{n-1}}P_{\lambda}(x'-y')\phi(y')\,dy'. 
\end{equation}
Observe that $\rho$ extends up to the boundary of ${\BBR}^{n}_{+}$ and maps one-to-one from 
$\partial {\BBR}^{n}_{+}$ onto $\partial\Omega$. Also for sufficiently small $\gamma\lesssim L$ 
the map $\rho$ is a bijection from $\overline{\mathbb{R}^{n}_{+} }$ onto $\overline\Omega$ 
and, hence, invertible. 

For a solution $u\in W^{1,2}_{loc}(\Omega;\BBC)$ to $\mathcal{L}u=0$ in $\Omega$ with Dirichlet 
datum $f$, consider $v:=u\circ\rho$ and $\tilde{f}:=f\circ\rho$. The change of variables 
via the map $\rho$ just described implies that $v\in W^{1,2}_{loc}(\mathbb{R}^{n}_{+};{\BBC})$ 
solves a new elliptic PDE of the form
\begin{equation}\label{ESvvv}
0=\mbox{div} (\tilde{A}(x)\nabla v)+\tilde{B}(x)\cdot\nabla v,
\end{equation}
with boundary datum $\tilde{f}$ on $\partial \mathbb{R}^{n}_{+}$. Hence, solving a boundary value 
problem for $u$ in $\Omega$ is equivalent to solving a related boundary value problem for $v$ in 
$\mathbb{R}^{n}_{+}$. Crucially, if the coefficients of the original system are such that \eqref{CarA} 
is a Carleson measure, then the coefficients of $\tilde{A}$ and $\tilde{B}$ satisfy an analogous 
Carleson condition in the upper-half space. If, in addition, the Carleson norm of \eqref{CarA} 
is small and $L$ (the Lipschitz constant for the domain $\Omega$) is also small, then the Carleson 
norm for the new coefficients $\tilde{A}$ and $\tilde{B}$ will be correspondingly small. Moreover, this transformation preserves $p$-ellipticity. Hence the 
map $\rho$ allows us to assume that the
domain is $\Omega={\BBR}^n_+$. 

\subsection{The $L^p$-Dirichlet problem}

We recall the definition of $L^p$ solvability of the Dirichlet problem. 
When an operator $\mathcal L$ is as in Theorem \ref{S3:T1} is uniformly elliptic (i.e. $2$-elliptic)
the Lax-Milgram lemma can be applied and guarantees the existence of weak solutions.
That is, 
given any $f\in \dot{B}^{2,2}_{1/2}(\partial\Omega;{\BBC})$, the homogenous space of traces of functions in $\dot{W}^{1,2}(\Omega;{\BBC})$, there exists a unique (up to a constant)
$u\in \dot{W}^{1,2}(\Omega;{\BBC})$ such that $\mathcal{L}u=0$ in $\Omega$ and ${\rm Tr}\,u=f$ on $\partial\Omega$. We call these solutions \lq\lq  energy solutions" and use them to define the notion of solvability of the $L^p$ Dirichlet problem.

\begin{definition}\label{D:Dirichlet} 
Let $\Omega$ be the Lipschitz domain introduced in \eqref{Omega-111} and fix an integrability exponent 
$p\in(1,\infty)$. Also, fix an aperture parameter $a>0$. Consider the following Dirichlet problem 
for a complex valued function $u:\Omega\to{\BBC}$:
\begin{equation}\label{E:D}
\begin{cases}
0=\partial_{i}\left(A_{ij}(x)\partial_{j}u\right) 
+B_{i}(x)\partial_{i}u 
& \text{in } \Omega,
\\[4pt]
u(x)=f(x) & \text{ for $\sigma$-a.e. }\,x\in\partial\Omega, 
\\[4pt]
\tilde{N}_{2,a}(u) \in L^{p}(\partial \Omega), &
\end{cases}
\end{equation}
where the usual Einstein summation convention over repeated indices ($i,j$ in this case) 
is employed. 

We say the Dirichlet problem \eqref{E:D} is solvable for a given $p\in(1,\infty)$ if  there exists a
$C=C(p,\Omega)>0$ such that
for all boundary data
$f\in L^p(\partial\Omega;{\BBC})\cap B^{2,2}_{1/2}(\partial\Omega;{\BBC})$ the unique energy solution
satisfies the estimate
\begin{equation}\label{y7tGV}
\|\tilde{N}_{2,a} (u)\|_{L^{p}(\partial\Omega)}\leq C\|f\|_{L^{p}(\partial\Omega;{\BBC})}.
\end{equation}
\end{definition}

\noindent{\it Remark.}  Given $f\in\dot{B}^{2,2}_{1/2}(\partial\Omega;{\BBC})\cap L^p(\partial\Omega;{\BBC})$
the corresponding energy solution constructed above is unique (since the decay implied by the $L^p$ estimates eliminates constant solutions). As the space 
$\dot{B}^{2,2}_{1/2}(\partial\Omega;{\BBC})\cap L^p(\partial\Omega;{\BBC})$ is dense in 
$L^p(\partial\Omega;{\BBC})$ for each $p\in(1,\infty)$, it follows that there exists a 
unique continuous extension of the solution operator
$f\mapsto u$
to the whole space $L^p(\partial\Omega;{\BBC})$, with $u$ such that $\tilde{N}_{2,a} (u)\in L^p(\partial\Omega)$ 
and the accompanying estimate $\|\tilde{N}_{2,a} (u) \|_{L^{p}(\partial \Omega)} 
\leq C\|f\|_{L^{p}(\partial\Omega;{\BBC})}$ being valid. Furthermore, as shown in \cite{DPcplx} for any $f\in L^p(\partial \Omega;\mathbb C)$ the corresponding solution $u$ constructed by the continuous extension attains the datum $f$ as its boundary values in the following sense.
Consider the average $\tilde u:\Omega\to \mathbb C$ defined by
$$\tilde{u}(x)=\dint_{B_{\delta(x)/2}(x)} u(y)\,dy,\quad \forall x\in \Omega.$$
Then 
\begin{equation}
f(Q)=\lim_{x\to Q,\,x\in\Gamma(Q)}\tilde u(x),\qquad\text{for a.e. }Q\in\partial\Omega,
\end{equation}
where the a.e. convergence is taken with respect to the ${\mathcal H}^{n-1}$ Hausdorff measure on $\partial\Omega$.

As we introduced in \cite{DPcplx}, the solutions to the Dirichlet problem in the infinite domain $\BBR^n_+$ will be obtained as a limit of solutions in
infinite strips $\Omega^h=\{x=(x_0, x') \in \BBR \times {\BBR}^{n-1}: 0 < x_0 < h\}$. These are defined as follows.

\begin{definition}\label{D:DirichletStrip} 
Let $\Omega = {\BBR}^n_+$, and let $\Omega^h$ be the infinite strip 
$$\Omega^h=\{x=(x_0, x') \in \BBR \times {\BBR}^{n-1}: 0 < x_0 < h\},$$ and let
$p\in(1,\infty)$. Also, fix an aperture parameter $a>0$. 
Let $u$ be a complex valued function $u:\Omega\to{\BBC}$ such that

\begin{equation}\label{E:D-strip}
\begin{cases}
0=\partial_{i}\left(A_{ij}(x)\partial_{j}u\right) 
+B_{i}(x)\partial_{i}u 
& \text{in } \Omega^h,
\\[4pt]
u(x_0,x') = 0, & \text{for all }x_0 \geq h,
\\[4pt]
u(x)=f(x) & \text{ for $\sigma$-a.e. }\,x\in\partial\Omega, 
\\[4pt]
\tilde{N}_{2,a}(u) \in L^{p}(\partial \Omega), &
\end{cases}
\end{equation}
where the usual Einstein summation convention over repeated indices ($i,j$ in this case) 
is employed. 

We say the Dirichlet problem \eqref{E:D-strip} is solvable for a given $p\in(1,\infty)$ if  there exists a
$C=C(p,\Omega)>0$ such that
for all boundary data
$f\in L^p(\partial\Omega;{\BBC})\cap B^{2,2}_{1/2}(\partial\Omega;{\BBC})$ we have that $u\big|_{\Omega^h}$ is the unique ``energy solution" to
\begin{equation}\label{E:D-energy}
\begin{cases}
0=\partial_{i}\left(A_{ij}(x)\partial_{j}u\right) 
+B_{i}(x)\partial_{i}u 
& \text{in } \Omega^h,
\\[4pt]
u(x_0,x') = 0, & \text{for } x_0=h
\\[4pt]
u(x)=f(x) & \text{ for $\sigma$-a.e. }\,x\in\partial\Omega, 
\end{cases}
\end{equation}

and satisfies the estimate
\begin{equation}\label{y7tGV-2}
\|\tilde{N}_{2,a} (u)\|_{L^{p}(\partial\Omega)}\leq C\|f\|_{L^{p}(\partial\Omega;{\BBC})}.
\end{equation}
\end{definition}

\vskip2mm


\section{Estimates for a $p$-adapted square function $S_p(u)$}
\label{S4}

In this section we establish some relationships between square functions and nontangential maximal functions that are
key to the proof of Theorem \ref{MainPerturb}. Let $\mathcal L_0$ and $\mathcal L_1$ satisfy the hypotheses of Theorem \ref{MainPerturb}; that is,
$\mathcal L_1$ is the perturbation of an operator $\mathcal L_0$ whose coefficients satisfy a Carleson measure condition.

We fix an $h>1$, and an infinite strip $\Omega^h$ defined in the previous section.
Let us assume that
$u$ is an energy solution to \eqref{E:D-energy}  in the strip for the operator $\mathcal L_1$ and extended to be zero above 
height $h$.  
In this section we establish an estimate of the $p$-adapted square function of $u$ in terms of boundary data and 
its nontangential maximal function. The constants appearing in the estimate will be independent of the height $h$.

\begin{lemma}\label{S3:L4}
Let $\Omega=\BBR^n_+$ and $\mathcal{L}_0$ and $\mathcal{L}_1$ are as in Theorem \ref{MainPerturb}.
Let $u:\Omega \to  \BBC$ be as above, 
with the Dirichlet boundary datum $f\in \dot{B}^{2,2}_{1/2}(\partial\Omega;{\BBC}) \cap  L^{p}(\partial \Omega;{\BBC})$.
Then there exists $K=K(\lambda_p,\Lambda,n,p)>0$ such that if 
$$\|m\|_{\mathcal C}<K$$
then for all $r>0$
\begin{align}\label{S3:L4:E00}
&
p\frac{\lambda_p}4\iint_{[0,r/2]\times\partial\Omega}|u|^{p-2}|\nabla u|^{2}x_0\,dx'\,d x_0 
+\frac{2}{r}\iint_{[0,r]\times\partial \Omega} |u(x_0,x')|^{p}\,dx'\,dx_0 
\nonumber\\[4pt]
&
\leq\int_{\partial\Omega}|u(0,x')|^{p}\,dx' 
+\int_{\partial\Omega}|u(r,x')|^{p}\,dx'
+C(\|\mu'\|_{\mathcal{C}}+\|m\|^{1/2}_{\mathcal{C}})\int_{\partial\Omega}\left[\tilde{N}^{r}_{p,a}(u)\right]^{p}\,dx'.
\end{align}
Here the constant on the righthand side $C=C(\lambda_p,\Lambda,n,p)>0$ is independent of $r$.
\end{lemma}

\begin{proof}  We extend the methods of Lemma 5.1 of \cite{DPcplx} to the case of an inhomogeneous equation.
Fix an arbitrary $y'\in\partial\Omega\equiv{\mathbb{R}}^{n-1}$, and consider first the case $r \leq h$. Choose 
a smooth cutoff function $\zeta$ which is $x_0-$independent and satisfies
\begin{equation}\label{cutoff-F}
\zeta= 
\begin{cases}
1 & \text{ in } B_{r}(y'), 
\\
0 & \text{ outside } B_{2r}(y').
\end{cases}
\end{equation}
Moreover, assume that $r|\nabla \zeta| \leq c$ for some positive constant $c$ independent of $y'$. 
Because the coefficients of operator $\mathcal{L}_0$ are differentiable, while the coefficients of $\mathcal{L}_1$ are not, we rewrite the solution of $\mathcal{L}_1u=0$ as follows:
\begin{equation}
\mathcal{L}_0u=\mbox{\rm div} \,\bold{F} + \beta \cdot \nabla u,\label{eq5.11}
\end{equation}
where $\bold{F}_i = \varepsilon_{ij} \partial_j u$ and $\varepsilon_{ij}=A^0_{ij}-A^1_{ij}$ and $\beta_j=B^0_j-B^1_j$.

We begin by considering the integral quantity 
\begin{equation}\label{A00}
\mathcal{I}:=\mathscr{R}e\,\iint_{[0,r]\times B_{2r}(y')}A^0_{ij}\partial_{j}u 
\partial_{i}(|u|^{p-2}\overline{u})x_0\zeta\,dx'\,dx_0
\end{equation}
with the usual summation convention understood. With $\chi=x_0\zeta$ we have, by $p$-ellipticity (c.f. Theorem 2.4 of \cite{DPcplx}),
for some $\lambda_p>0$
\begin{equation}\label{cutoff-AA}
\mathcal{I}\geq{\lambda_p}\iint_{[0,r]\times B_{2r}}|u|^{p-2}|\nabla u|^2 x_0\zeta\,dx'\,dx_0,
\end{equation}
where we use the abbreviation $B_{2r}:=B_{2r}(y')$ whenever convenient.

We integrate by parts the formula for $\mathcal I$ in order 
to relocate the $\partial_i$ derivative. This gives 
\begin{align}\label{I+...+IV}
\mathcal{I}
&= \mathscr{R}e\,\int_{\partial\left[(0,r)\times B_{2r}\right]} 
A^0_{ij}\partial_{j}u|u|^{p-2}\overline{u}x_0\zeta\nu_{x_i}\,d\sigma 
\nonumber\\[4pt]
&\quad - \mathscr{R}e\,\iint_{[0,r]\times B_{2r}}\partial_{i}\left(A^0_{ij} 
\partial_{j}u\right)|u|^{p-2}\overline{u}x_0\zeta\,dx'\,dx_0 
\nonumber\\[4pt]
&\quad - \mathscr{R}e\,\iint_{[0,r]\times B_{2r}}A^0_{ij}\partial_{j}{u}|u|^{p-2}\overline{u}\partial_{i}x_0\zeta\,dx'\,dx_0 
\nonumber\\[4pt]
&\quad - \mathscr{R}e\,\iint_{[0,r]\times B_{2r}}A^0_{ij}\partial_{j}u|u|^{p-2}\overline{u}x_0\partial_{i}\zeta\,dx'\,dx_0
\nonumber\\[4pt]
&=:I+II+III+IV,
\end{align}
where $\nu$ is the outer unit normal vector to the domain $(0,r)\times B_{2r}$. The boundary term $I$ vanishes
except on the set $\{r\}\times B_{2r}$, and only when $i=0$. This gives
\begin{equation}\label{cutoff-BBB}
I=\mathscr{R}e\,\int_{\{r\}\times B_{2r}} 
A^0_{0j}\partial_{j}u|u|^{p-2}\overline{u}x_0\zeta\,d\sigma 
\end{equation}

As $u$ is a weak solution of  \eqref{eq5.11} in $\Omega$, we use the equation to transform $II$ into
\begin{align}\label{cutoff-CCC}
II &=\mathscr{R}e\,\iint_{[0,r]\times B_{2r}}(B^0_{i}-
\beta_i)(\partial_{i}u)|u|^{p-2}\overline{u}x_0\zeta\,dx'\,dx_0
\nonumber\\[4pt]
&\quad  - \mathscr{R}e\,\iint_{[0,r]\times B_{2r}} \partial_{i} \left(\varepsilon_{ij} 
\partial_{j}u \right) |u|^{p-2}\overline{u}x_0\zeta\,dx'\,dx_0 
\nonumber\\[4pt]
&\quad =: II_1 + II_2.
\end{align}

To estimate the first term on the right hand side, we use H\"older's inequality, the Carleson condition of Theorem \ref{S3:T1} for the term $B^0$, the Carleson condition \eqref{eqdm} and 
Theorem~\ref{T:Car} to see that
\begin{align}\label{TWO-ONE}
|II_1| &\leq\left(\iint_{[0,r]\times B_{2r}} (|B^0|^{2}+|\beta|^2) 
|u|^{p} x_0\zeta\,dx'\,dx_0\right)^{1/2}  \times
\nonumber\\[4pt]
&\qquad\left(\iint_{[0,r]\times B_{2r}}|u|^{p-2}|\partial_{j}u|^{2}x_0\zeta\,dx'\,dx_0\right)^{1/2} 
\nonumber\\[4pt]
&\leq C(\lambda_p,\Lambda,N)\left((\|\mu'\|_{\mathcal{C}}+\|m\|_{\mathcal{C}})\int_{B_{2r}} 
\left[\tilde{N}^r_{p,a}(u)\right]^{p}\,dx'\right)^{1/2}\cdot\mathcal{I}^{1/2}. 
\end{align}

We integrate term $II_2$ by parts, obtaining a boundary term when $i=0$:

\begin{align}
II_2 &= - \mathscr{R}e\,\iint_{[0,r]\times B_{2r}} \partial_{i} \left(\varepsilon_{ij} 
\partial_{j}u \right) |u|^{p-2}\overline{u}x_0\zeta\,dx'\,dx_0 
\nonumber\\[4pt]
& =
- \mathscr{R}e\,\iint_{[0,r]\times B_{2r}} \varepsilon_{ij} 
\partial_{j}u   \partial_{i}\left(|u|^{p-2}\overline{u}x_0\zeta\,\right)dx'\,dx_0 
\nonumber\\[4pt]
&+ \mathscr{R}e\,\int_{ B_{2r}} r \varepsilon_{0j} 
\partial_{j}u(x',r) |u(x',r)|^{p-2}\overline{u(x,r)}\zeta\,dx'
\nonumber\\[4pt]
&=- \mathscr{R}e\,\iint_{[0,r]\times B_{2r}} \varepsilon_{ij} 
\partial_{j}u   \partial_{i}\left(|u|^{p-2}\overline{u}\right)x_0\zeta\,dx'\,dx_0 
\nonumber\\[4pt]
&- \mathscr{R}e\,\iint_{[0,r]\times B_{2r}} \varepsilon_{0j} 
\partial_{j}u |u|^{p-2}\overline{u}\zeta\,dx'\,dx_0 
\nonumber\\[4pt]
&- \mathscr{R}e\,\iint_{[0,r]\times B_{2r}} \varepsilon_{ij} 
\partial_{j}u |u|^{p-2}\overline{u}x_0\partial_{i}\zeta\,dx'\,dx_0 
\nonumber\\[4pt]
&+ \mathscr{R}e\,\int_{ B_{2r}} r \varepsilon_{0j} 
\partial_{j}u(x',r) |u(x',r)|^{p-2}\overline{u(x,r)}\zeta\,dx'
\nonumber\\[4pt]
&=II_{21}+II_{22}+II_{23}+II_{24}.
\end{align}

For the term $II_{21}$ we have the estimate
\begin{align}
|II_{21}| &\lesssim 
\iint_{[0,r]\times B_{2r}} |\varepsilon_{ij}| 
|\partial_{j}u|   |\partial_{i}u||u|^{p-2}x_0\zeta\,dx'\,dx_0 
\nonumber\\[4pt]
&
\le \sup_{i,j}\|\varepsilon_{ij}\|_{L^\infty}\,\mathcal{I}\lesssim \|m\|^{1/2}_{\mathcal C}\,\mathcal{I}.
\end{align}
the last estimate is a consequence of \eqref{eqdm}, since the Carleson condition implies $L^\infty$ bounds on $\varepsilon_{ij}$. Similarly, for term $II_{22}$ we have by Cauchy-Schwarz and the Carleson condition \eqref{eqdm}

\begin{align}
|II_{22}| &\lesssim 
\left(\iint_{[0,r]\times B_{2r}}  
|\partial_{j}u|^2   |u|^{p-2}x_0\zeta\,dx'\,dx_0 \right)^{1/2}\left(\iint_{[0,r]\times B_{2r}}  
\frac{|\varepsilon_{ij}|^2}{x_0}   |u|^{p}\zeta\,dx'\,dx_0 \right)^{1/2}.
\nonumber\\[4pt]
&\leq C(\lambda_p,\Lambda,p,n)\left(\|m\|_{\mathcal{C}}\int_{B_{2r}}
\left[\tilde{N}^r_{p,a}(u)\right]^{p}\,dx'\right)^{1/2}\cdot\mathcal{I}^{1/2}. 
\nonumber\\[4pt]
\end{align}
We shall deal with the last two terms $II_{23}$ and $II_{24}$ later.

As $\partial_ix_0=0$ for $i>0$ the term $III$ is non-vanishing only for $i=0$. We further split this term
by separately considering the cases when $j=0$ and $j>0$. This yields, since $A_{00} =1$, 

\begin{align}\label{u6fF}
III_{\{j=0\}} &=-\mathscr{R}e\,\iint_{[0,r]\times B_{2r}}\partial_{0}{u}|u|^{p-2}\overline{u}\zeta\,dx'\,dx_0 \nonumber\\
&=-\frac1p\iint_{[0,r]\times B_{2r}}  \partial_{0}(|u|^{p})\zeta\,dx'\,dx_0 \\
&=-\frac{1}{p}\int_{B_{2r}} |u|^p(r,x')\zeta\,dx' + \frac{1}{p}\int_{B_{2r}} |u|^p(0,x')\zeta\,dx' \nonumber
\end{align}

When $j>0$ we first use the fact that $A_{0j}^0$ is real and hence the expression $\mathscr{R}e\, [{A}_{0j}^0\, (\partial_{j}u)|u|^{p-2}\overline u]=p^{-1}A_{0j}^0\partial_j(|u|^p)$. Then we 
reintroduce $1=\partial_0x_0$  and integrate by parts moving the $\partial_0$ derivative
\[\begin{split}
III_{\{j \neq 0\}}
&= -\mathscr{R}e\, \iint_{[0,r]\times B_{2r}} {A}_{0j}^0\, \partial_{j}u|u|^{p-2} \, \overline{u} \,  \zeta \,dx'\,dx_0 \\
& -p^{-1}\, \iint_{[0,r]\times B_{2r}} {A}^0_{0j}\, \partial_{j}(|u|^{p}) \, \left(\partial_{0} x_0\right) \zeta \,dx'\,dx_0 \\
&= p^{-1}\int_{B_{2r}} {A}^0_{0j} \partial_{j}(|u|^p)(r,x') r \zeta \,dx' + p^{-1}\iint_{[0,r]\times B_{2r}} \partial_{0}{A}_{0j}^0 \partial_{j}(|u|^p) x_0 \zeta \,dx'\,dx_0 \\
&+ p^{-1}\iint_{[0,r]\times B_{2r}} {A}_{0j}^0 \partial^2_{0j}(|u|^p) x_0 \zeta \,dx'\,dx_0  \\
&= III_{1} + III_{2} + III_{3}.
\end{split}\]
We note that $III_{1} = - I_{\{j \neq 0\}}$.

In the third term $III_3$ we switch the order of derivatives $\partial^2_{0j} = \partial^2_{j0}$ and make a further integration by parts with respect to $\partial_j$.
\[\begin{split}
III_{3}
&= - p^{-1}\iint_{[0,r]\times B_{2r}} \partial_{j} {A}_{0j}^0 \partial_{0}(|u|^p) x_0 \zeta \,dx'\,dx_0 \\
&\quad - p^{-1}\iint_{[0,r]\times B_{2r}} {A}_{0j}^0\partial_{0}(|u|^p) x_0 (\partial_{j}\zeta) \,dx'\,dx_0 = III_{31} + III_{32}.
\end{split}\]

The terms $III_2$ and $III_{31}$ are of similar type as $II_{1}$ we have handled earlier and hence have the same estimate
\[
III_{2} + III_{31} \leq C(\lambda_p,\Lambda, p,n) \left( \|\mu'\|_{\mathcal{C}} \int_{ B_{2r}} \left[\tilde{N}^r_{p,a}(u)\right]^{p}\,dx'\right)^{1/2} \cdot \mathcal I^{1/2}
\]

We add up all terms we have so far to obtain
\begin{equation}\label{square01}\begin{split}
\mathcal I
&\leq p^{-1}\int_{B_{2r}}  \partial_{0}(|u|^p)(r,x') r \zeta \,dx' \\
&\quad  - {p}^{-1}\int_{B_{2r}} |u|^p(r,x')\zeta\,dx' + {p}^{-1}\int_{B_{2r}} |u|^p(0,x')\zeta\,dx'  \\
&\quad + C(\lambda_p,\Lambda,p,n) (\|\mu'\|_{\mathcal{C}}+\|m\|_{\mathcal{C}}) \int_{B_{2r}} \left[\tilde{N}^{r}_{p,a}(u)\right]^p(u) \,dx' \\
&\quad+\left(\frac14+C'(\lambda_p,\Lambda,p,n)\|m\|^{1/2}_{\mathcal{C}}\right)\mathcal I \\
&\quad + II_{23}+II_{24}+ III_{32}  + IV.
\end{split}\end{equation}

We have used the arithmetic-geometric inequality for the expression bounding the terms $II_1$, $II_{22}$ as well as for similar terms $III_2$ and $III_{31}$. Observe that if $C'(\lambda_p,\Lambda,p,n)\|m\|^{1/2}<1/4$ the term containing $\mathcal I$ can be absorbed by the lefthand side of \eqref{square01}.

To obtain a global version of \eqref{square01}, consider a sequence of disjoint boundary balls 
$(B_r(y'_k))_{k\in\mathbb N}$ such that $\cup_{k}B_{2r}(y'_k) $ covers $\partial\Omega={\BBR}^{n-1}$ 
and consider a partition of unity $(\zeta_{k})_{k\in\mathbb N}$ subordinate to this cover. That is, 
assume $\sum_k \zeta_{k} = 1$ on ${\BBR}^{n-1}$ and each $\zeta_{k}$ is supported in $B_{2r}(y'_k)$. 
Write $IV_k$ for each term as the last expression in \eqref{I+...+IV} corresponding to 
$B_{2r}=B_{2r}(y'_k)$. Given that $\sum_k \partial_i\zeta_{k} = 0$ for each $i$, by summing 
\eqref{square01} over all $k$'s gives $\sum_{k} IV_k= 0$. The same observation applies to the terms arising in $II_{23}$ and $III_{32}$.
It follows that
\begin{align}\label{square02}
&\hskip -0.20in 
\iint_{[0,r]\times{\BBR}^{n-1}}|\nabla u|^2|u|^{p-2}\,x_0\,dx'\,dx_0 
\nonumber\\[4pt]
\lesssim&\hskip 0.20in
p^{-1}\int_{{\BBR}^{n-1}}  \partial_{0}(|u|^p)(r,x') r  \,dx' \nonumber\\
&\quad  - {p}^{-1}\int_{{\BBR}^{n-1}} |u|^p(r,x')\,dx' + {p}^{-1}\int_{{\BBR}^{n-1}} |u|^p(0,x')\,dx'  
\nonumber\\[4pt]
&\quad +C(\|\mu'\|_{\mathcal{C}}+\|m\|_{\mathcal{C}})\int_{{\BBR}^{n-1}}\left[\tilde{N}^{r}_{p,a}(u)\right]^p\,dx'
\nonumber\\[4pt]
&\quad+ \mathscr{R}e\,\int_{ {\mathbb R}^{n-1}} r \varepsilon_{0j} 
\partial_{j}u(x',r) |u(x',r)|^{p-2}\overline{u(x,r)}\,dx'.
\end{align}

We have established \eqref{square02} for $r \leq h$, but just as in \cite{DPcplx}, section 5, we can
see that \eqref{square02} holds also for $r>h$ since $u = 0$ when $r \geq h$.
Then, \eqref{S3:L4:E00} follows by integrating \eqref{square02} in $r$ over
$[0,r']$ and dividing by $r'$.  The estimate for the last term of \eqref{square02} requires the Carleson condition on the difference of the
the coefficients and we argue as follows. After the integration and averaging we obtain the quantity
\begin{align}\label{square022}
& \left|\mathscr{R}e\,\frac1{r'}\iint_{ {\mathbb R}^{n-1}\times[0,r']}  \varepsilon_{0j} 
\partial_{j}u(x',x_0) |u(x',r)|^{p-2}\overline{u(x,x_0)}x_0\,dx'\,dx_0\right|
\nonumber\\[4pt]
\le&\,\left(\iint_{ {\mathbb R}^{n-1}\times[0,r']}  \frac{|\varepsilon_{0j}|^2}{x_0} |u|^p\,dx'\,dx_0\right)^{1/2}
\left(\iint_{ {\mathbb R}^{n-1}\times[0,r']} |\partial_{j}u|^2 |u|^{p-2}x_0\,dx'\,dx_0\right)^{1/2}
\nonumber\\[4pt]
\le&\,C(\lambda_p,\Lambda, p,n) \left( \|m\|_{\mathcal{C}} \int_{ B_{2r}} \left[\tilde{N}^r_{p,a}(u)\right]^{p}\,dx'\right)^{1/2}
\nonumber\\[4pt]
&\qquad\times\left(\iint_{ {\mathbb R}^{n-1}\times[0,r']} |\partial_{j}u|^2 |u|^{p-2}x_0\,dx'\,dx_0\right)^{1/2}.
\end{align}
The last term of \eqref{square022} can be split as a sum of two terms: an integral over
${\mathbb R}^{n-1}\times[0,r'/2]$ and an integral over ${\mathbb R}^{n-1}\times[r'/2,r']$.
The integral over ${\mathbb R}^{n-1}\times[0,r'/2]$ appears on the righthand side of \eqref{S3:L4:E00} and hence can be absorbed by it. For the second integral we use Theorem 1.1 and Proposition 3.5 from \cite{DPcplx} to obtain
\begin{align}\label{square023}
&\iint_{ {\mathbb R}^{n-1}\times[r'/2,r']} |\partial_{j}u|^2 |u|^{p-2}x_0\,dx'\,dx_0
\lesssim \int_{ {\mathbb R}^{n-1}} \left[\tilde{N}^{r}_{p,a}(u)\right]^{p}\,dx'.
\end{align}
From this the claim follows.
\end{proof}

Lemma~\ref{S3:L4} has two important corollaries. Their proofs are not short, but do not differ in any respect from the corresponding
proofs of Corollaries 5.2 and 5.5 of \cite{DPcplx}, except that certain constants now depend on $\|m\|^{1/2}_{\mathcal{C}}$.

\begin{corollary}\label{S4:C2} 
Under the assumptions of Lemma \ref{S3:L4} we have for such $u$:
\begin{eqnarray}\label{S3:L4:E00bb}
\lambda'_p\iint_{\BBR^n_+}|\nabla u|^{2}|u|^{p-2}x_0\,dx'\,d x_0 
&\leq&\int_{\BBR^{n-1}}|u(0,x')|^{p}\,dx'
+\\\nonumber
&&C(\|\mu'\|_{\mathcal{C}}+\|m\|^{1/2}_{\mathcal{C}})\int_{\BBR^{n-1}}\left[\tilde{N}_{p,a}(u)\right]^{p}\,dx'.
\end{eqnarray}

Furthermore, under the same assumptions, if $g:{\mathbb R}^{n-1}\to{\mathbb R}^+$ is a Lipschitz function with small Lipschitz norm for any $\Delta\subset{\mathbb R}^{n-1}$ such that $\sup_\Delta g\le d/2$ where $d=\mbox{diam}(\Delta)$ we also have the following local estimate

\begin{eqnarray}\label{eq5.15}
\iint_{\Omega_g\cap T(\Delta)}|\nabla u|^{2}|u|^{p-2}\delta_g(x)\,dx 
&\leq &C\int_{2\Delta}|u(g(x'),x')|^p+\\\nonumber 
&&C(1+\|\mu\|_{\mathcal{C}} + \|m\|^{1/2}_{\mathcal{C}} ) \int_{2\Delta}\left[\tilde{N}^{2d}_{p,a,g}(u)\right]^{p}\,dx'.
\end{eqnarray}
Here $\tilde{N}^{2d}_{p,a,g}$ is the truncated version of the nontangential maximal function defined in \eqref{eq-Nh}
with respect to the domain $\Omega_g=\{x_0>g(x')\}$ and $\delta_g$ measures the distance of a point to the boundary of $\Omega_g$. 
\end{corollary}

\begin{corollary}\label{S4:C1} Under the assumption of Lemma \ref{S3:L4}, for any $q \geq p >1$ and $a>0$ there exists a finite 
constant $C=C(\lambda_p,\Lambda,p,q,a,\|\mu\|_{\mathcal C},n)>0$ such that 
\begin{equation}\label{S3:C7:E00oo=s}
\|S_{p,a}(u)\|_{L^{q}({\BBR}^{n-1})}\le C\|\tilde{N}_{p,a}(u)\|_{L^{q}({\BBR}^{n-1})}.
\end{equation}
The statement also holds for any $q > 0$, 
provided we know a priori that
\newline
$\|S_{p,a}(u)\|_{L^{q}({\BBR}^{n-1})} < \infty.$

\end{corollary}

\section{Bounds for the nontangential maximal function by the $p$-adapted square function}
\label{SS:43}

The following proposition is the analog of Corollary 6.2 of 
\cite{DPcplx}. That Corollary was, in turn, a consequence of an estimate on the  $L^q$ norm of the nontangential maximal functions, $\tilde{N_2}$, by  
the $L^q$ norm of the square function, $S_2$, proved in \cite{DHM} for non-symmetric systems whose coefficients satisfied Carleson measure conditions.
Our perturbed operator $\mathcal L_1$ doesn't satisfy these conditions, and we will not have recourse to this fact, namely Proposition 5.8 of \cite{DHM}.
Therefore, we have to prove that Proposition 5.8 of \cite{DHM} holds under perturbations of the type we are considering.

\begin{proposition}\label{S3:C7} 
Let $1<p<\infty$, $\Omega=\BBR^n_+$ and $\mathcal{L}_0$ and $\mathcal{L}_1$ be as in Theorem \ref{MainPerturb}.
Assume that
$u$ is an energy solution to \eqref{E:D-energy}  in the strip for the operator $\mathcal L_1$, 
with the Dirichlet boundary datum $f\in \dot{B}^{2,2}_{1/2}(\partial\Omega;{\BBC}) \cap  L^{p}(\partial \Omega;{\BBC})$
and extended to be zero above 
height $h$.
Assume that  the measures $\mu$ defined by 
\eqref{Car_hatAA} and the measure $m$ defined by \eqref{eqdm1} are Carleson with finite norms.
There exists $K=K(\lambda_p,\Lambda,n,p)>0$ such that if 
$\|m\|_{\mathcal C}<K$
then for any $q>0$ and $a>0$ there exists a positive
constant $C=C(\lambda_p,\Lambda,p,q,a,n,\|\mu\|_{\mathcal C})$ such that 
\begin{equation}\label{S3:C7:E00oo}
\|\tilde{N}_{p,a}(u)\|_{L^{q}({\BBR}^{n-1})}\le C\|S_{p,a}(u)\|_{L^{q}({\BBR}^{n-1})},
\end{equation}
provided that a priori $\|\tilde{N}_{p,a}(u)\|_{L^{q}({\BBR}^{n-1})} < \infty$. If the dual exponent $p'>q$ we also have to assume that $\|S_{p',a}(u)\|_{L^{q}({\BBR}^{n-1})} < \infty$.

\end{proposition}

The only missing ingredient is the following analogue of Lemma 5.4 of \cite{DHM}.

\begin{lemma}\label{S3:L8} Let $\Omega=\BBR^n_+$ and $\mathcal{L}_0$ and $\mathcal{L}_1$ be as in Theorem \ref{MainPerturb}.
Assume $u$ is as in Proposition \ref{S3:C7}.  Then there exists $a>0$ with the following significance.  For any $\theta\in[1/6,6]$ if $\phi:{\mathbb R}^{n-1}\to\mathbb R$ is a Lipschitz function with Lipschitz constant $1/a$ we consider the domain 
$\mathcal{O}=\{(x_0,x')\in\Omega:\,x_0>\theta \phi(x')\}$ with boundary  
$\partial\mathcal{O}=\{(x_0,x')\in\Omega:\,x_0=\theta \phi(x')\}$. In this context, 
for any surface ball $\Delta_r=B_r(Q)\cap\partial\Omega$, with $Q\in\partial\Omega$ and $r>0$ 
chosen such that $\phi \leq 2r$ pointwise on $\Delta_{2r}$, 
one has
\begin{align}\label{TTBBMM}
\int_{\Delta_r}\big|u\big(\theta \phi(\cdot),\cdot\big)\big|^2\,dx' 
&\leq C(1+\|\mu\|^{1/2}_{\mathcal C}+\|m\|^{1/2}_{\mathcal C})\Big[\|S_{2,b}(u)\|_{L^2(\Delta_{2r})}
\|\tilde{N}_{2,a}(u)\|_{L^2(\Delta_{2r})}
\nonumber\\
&\quad+\|S_{2,b}(u)\|^2_{L^2(\Delta_{2r})}\Big]+\frac{c}{r}\iint_{\mathcal{K}}|u|^{2}\,dX.
\end{align}
Here $C=C(\Lambda,n)\in(0,\infty)$ and $\mathcal{K}$ is a region inside $\mathcal{O}$ with diameter, 
distance to the boundary $\partial\mathcal{O}$, and distance to $Q$, all comparable to $r$. 
Also, the parameter $b>a$ is as in Lemma  5.2 of \cite{DHM}, and the cones used to define the square and nontangential 
maximal functions in this lemma have vertices on $\partial\Omega$.

Moreover, the term $\displaystyle\iint_{\mathcal{K}}|u|^2\,dX$ appearing 
in \eqref{TTBBMM} may be replaced by the quantity
\begin{equation}\label{Eqqq-25}
Cr^{n-1}|\tilde{u}(A_r)|^2+C\int_{\Delta_{2r}}S^2_{2,b}(u)\,d\sigma,
\end{equation}
where $A_r$ is any point inside $\mathcal{K}$ (usually referred to as a corkscrew point of $\Delta_r$) and
\begin{equation}\label{Eqqq-26}
\tilde{u}(X):=\dint_{B_{\delta(X)/2}(X)} u(Z)\,dZ.
\end{equation}
\end{lemma}

\begin{proof} 
Fix $\theta\in [1/6,6]$. We first consider the case when $r$ is small, i.e., $r\le 2h$. Consider the pullback transformation $\rho:\BBR^{n}_{+}\to\mathcal{O}$ 
defined as in section \ref{SS:PT} relative to the Lipschitz function $\phi$.
Let $v$ be given by $v:=u\circ\rho$ in $\BBR^{n}_{+}$. 
If $u$ satisfies \eqref{E:D-energy}, then the function 
$v:\BBR^{n}_+\to\BBR$ will satisfy a PDE similar to that of $u$. Specifically, we will have
\begin{equation}\label{ESv}
\partial_{i}\left(\bar{A}_{ij}(x)\partial_{j}v\right)
+\bar{B}_{i}(x)\partial_{i}v=0,
\end{equation}
where $\bar{A}$ is $p$-elliptic if the original $A$ was $p$-elliptic. Consider such pullback for solutions of both operators ${\mathcal L}_0$ and ${\mathcal L}_1$. Under the assumptions of Theorem \ref{S3:T1} for ${\mathcal L}_0$ we have that the coefficients $\bar{A}^0$ and $\bar{B}^0$ are such that
\begin{equation}\label{CarbarA}
d\overline{\mu}(x)=\left[\left(\sup_{B_{\delta(x)/2}(x)}|\nabla\bar{A}^0|\right)^{2}
+\left(\sup_{B_{\delta(x)/2}(x)}|\bar{B}^0|\right)^{2}\right]\delta(x) \,dx
\end{equation}
is a Carleson measure in ${\mathbb{R}}^n_{+}$. In virtue of the Carleson condition on
the difference of the coefficients of ${\mathcal L}_0$ and ${\mathcal L}_1$,  the coefficients  $\bar{A}^1$ and $\bar{B}^1$ of  the pullback of ${\mathcal L}_1$
satisfy
\begin{equation}\label{eqdm2}
 d\bar m(x)=\sup_{B_{\delta(x)/2}(x)}\left[|\bar A^0 - \bar A^1|^{2} \delta^{-1}(x)+ |\bar B^0 - \bar B^1|^{2} \delta(x)\right]\,dx
\end{equation}
is a Carleson measure in ${\mathbb{R}}^n_{+}$ with Carleson norm
$$\|\bar m\|_{\mathcal C}\le C(\|m\|_{\mathcal C},a), \qquad\mbox{and}\qquad C(\|m\|_{\mathcal C},a)\to \|m\|_{\mathcal C} \mbox{ as } a\to\infty.$$

Similarly, the Carleson norm $\|\overline\mu\|_{\mathcal{C}}$ also
only depends on the Carleson norm of the original coefficients and $a$  (the aperture of nontangential cones).
In particular, we can choose the parameter $a>0$ large enough such that $\|\bar m\|_{\mathcal C}$ and $\|\bar \mu\|_{\mathcal C}$ are at most 
twice that of $\|m\|_{\mathcal C}$ and $\|\mu\|_{\mathcal C}$, respectively.

We may also assume that the coefficient $\bar{A}_{00}=1$. This follows from a change of variables that modifies the lower order terms; see the
discussion following Definition 4.1 of \cite{DPcplx}.
\vskip1mm

Having fixed a scale $r>0$, we localize to a ball $B_r(y')$ in $\BBR^{n-1}$. 
Let $\zeta$ be a smooth cutoff function of the form $\zeta(x_0, x')=\zeta_{0}(x_0)\zeta_{1}(x')$ where
\begin{equation}\label{Eqqq-27}
\zeta_{0}= 
\begin{cases}
1 & \text{ in } [0, r], 
\\
0 & \text{ in } [2r, \infty),
\end{cases}
\qquad
\zeta_{1}= 
\begin{cases}
1 & \text{ in } B_{r}(y'), 
\\
0 & \text{ in } \mathbb{R}^{n}\setminus B_{2r}(y')
\end{cases}
\end{equation}
and
\begin{equation}\label{Eqqq-28}
r|\partial_{0}\zeta_{0}|+r|\nabla_{x'}\zeta_{1}|\leq c
\end{equation}
for some constant $c\in(0,\infty)$ independent of $r$. 

Let ${\mathcal L}_1u=0$. Our goal is to control the $L^p$ norm of $u\big(\theta h(\cdot),\cdot\big)$.  
Since after the pullback under the mapping $\rho$ the latter is comparable with the $L^p$ norm 
of $v(0,\cdot)$, we proceed to estimate
\begin{align}\label{u6tg}
&\hskip -0.20in
\int_{B_{2r}(y')}|v|^{2}(0,x')\zeta(0,x')\,dx' 
\nonumber\\[4pt]
&\hskip 0.70in
=-\iint_{[0,2r]\times B_{2r}(y')}\partial_{0}\left[|v|^{2}\zeta\right](x_0,x')\,dx_0\,dx' 
\nonumber\\[4pt]
&\hskip 0.70in
=-p\iint_{[0,2r]\times B_{2r}(y')} \mathscr{R}e\,\langle v,\partial_0 v\rangle\zeta\,dx_0\,dx'  
\nonumber\\[4pt]
&\hskip 0.70in
\quad-\iint_{[0,2r]\times B_{2r}(y')}|v|^{2}(x_0,x')\partial_{0}\zeta\,dx_0\,dx'
\nonumber\\[4pt]
&\hskip 0.70in
=:\mathcal{A}+IV.
\end{align}
We further expand the term $\mathcal A$ as a sum of four terms obtained 
via integration by parts with respect to $x_0$ as follows:
\begin{align}\label{utAA}
\mathcal A &=-2\iint_{[0,2r]\times B_{2r}(y')} \mathscr{R}e\,\langle v,\partial_0 v\rangle\zeta(\partial_0x_0)\,dx_0\,dx'
\nonumber\\[4pt]
&\quad=2\iint_{[0,2r]\times B_{2r}(y')} \left|\partial_{0}v\right|^{2}x_0\zeta\,dx_0\,dx' 
\nonumber\\[4pt]
&\quad +2\iint_{[0,2r]\times B_{2r}(y')} \mathscr{R}e\,\langle v,\partial^2_{00} v\rangle x_0\zeta\,dx_0\,dx' 
\nonumber\\[4pt]
&\quad +2\iint_{[0,2r]\times B_{2r}(y')} \mathscr{R}e\,\langle v,\partial_0 v\rangle x_0\partial_{0}\zeta\,dx_0\,dx' 
\nonumber\\[4pt]
&=:I+II+III.
\end{align}

We start by analyzing the term $II$. In view of the fact that $\bar{A}^0_{00}=1$, we can rewrite the equation for $v$
following \eqref{eq5.11}
\begin{equation}\label{S3:T8:E01}
\partial^2_{00}v
=-\sum_{(i,j)\neq(0,0)}\partial_{i}\left(\bar{A}^0_{ij}\partial_{j}v\right)
-(\bar B^0_i-\bar\beta^0_i)\partial_iv+\partial_{i}\left(\bar{\varepsilon}_{ij}\partial_{j}v\right),
\end{equation}
where $\bar\varepsilon_{ij}=\bar A^0_{ij}-\bar A^1_{ij}$ and $\bar\beta_j=\bar B^0_j-\bar B^1_j$.
In turn, this permits us to express
\begin{align}\label{TFWW}
II &=-2\,\mathscr{R}e\sum_{(i,j)\neq(0,0)}\iint_{[0,2r]\times B_{2r}}\left(\partial_{i}\bar{A}^0_{ij}\right) 
\overline{v}\partial_{j}vx_0\zeta\,dx_0\,dx' 
\nonumber\\[4pt]
&\quad -2\,\mathscr{R}e\iint_{[0,2r]\times B_{2r}}(\bar B^0_i-\bar\beta^0_i) \overline{v}\partial_{i}vx_0\zeta\,dx_0\,dx' 
\nonumber\\[4pt]
&\quad -2\,\mathscr{R}e\sum_{(i,j)\neq(0,0)}\iint_{[0,2r]\times B_{2r}}\bar{A}^0_{ij} \overline{v}\partial^2_{ij}vx_0 \zeta\,dx_0\,dx' 
\nonumber\\[4pt]
&\quad +2\,\mathscr{R}e\iint_{[0,2r]\times B_{2r}}\partial_{i}\left(\bar{\varepsilon}_{ij}\partial_{j}v\right) 
\overline{v}x_0\zeta\,dx_0\,dx'
\nonumber\\[4pt]
&=:II_{1}+II_{2}+II_{3}+II_4.
\end{align}
The third term above requires some further work. Let us temporarily fix $i,j$ and denote by 
$II_{3}^{ij}$ the corresponding term in $II_3$. Since in the present context we have
$(i,j)\neq(0,0)$, at least one of the two indices involved is not zero, say $i>0$. 
Integrating by parts with respect to the variable $x_i$ then yields (in what follows we do 
not sum over indices $i$ and $j$)
\begin{align}\label{6GBBB}
II_{3}^{ij} &=2\,\mathscr{R}e\iint_{[0,2r]\times B_{2r}}\left(\partial_{i}\bar{A}^0_{ij}\right) 
\overline{v}\partial_{j}vx_0\zeta\,dx_0\,dx' 
\nonumber\\[4pt]
&\quad +2\,\mathscr{R}e\iint_{[0,2r]\times B_{2r}}\bar{A}^0_{ij}\partial_{i}\left(\overline{v}\right)\partial_jv
x_0\zeta\,dx_0\,dx' 
\nonumber\\[4pt]
&\quad +2\,\mathscr{R}e\iint_{[0,2r]\times B_{2r}}\bar{A}^0_{ij} \overline{v}\partial_{j}v 
x_0\partial_{i}\zeta\,dx_0\,dx' 
\nonumber\\[4pt]
&=J^{ij}_1+J^{ij}_2+J^{ij}_3.
\end{align}
The treatment of $II_{3}^{ij}$ in the case when $i=0$ proceeds along the same lines, except that 
we now integrate in the variable $x_j$. Since the resulting terms are of a similar nature as above, 
we omit writing them explicitly. 

It remain to deal with the term $II_4$. We integrate by parts in the variable $i$. Due to the presence of the cutoff function there are no boundary terms. We obtain:

\begin{align}\label{III4}
II_4 &=-2\,\mathscr{R}e\iint_{[0,2r]\times B_{2r}}\bar{\varepsilon}_{ij}\partial_{j}v 
\partial_{i}(\overline{v})x_0\zeta\,dx_0\,dx'
\nonumber\\[4pt]
&\quad -2\,\mathscr{R}e\iint_{[0,2r]\times B_{2r}}\bar{\varepsilon}_{0j}\partial_{j}v 
\overline{v}\zeta\,dx_0\,dx'
\nonumber\\[4pt]
&\quad -2\,\mathscr{R}e\iint_{[0,2r]\times B_{2r}}\bar{\varepsilon}_{ij}\partial_{j}v 
\overline{v}x_0\partial_i\zeta\,dx_0\,dx' 
\nonumber\\[4pt]
&=II_{41}+II_{42}+II_{43}.
\nonumber\\[4pt]
\end{align}

We now group together terms that are of the same type. Firstly, we have 
\begin{equation}\label{Eqqq-29}
I+J_{2}\leq C(\Lambda,n)\|S_{2,b}(u)\|^2_{L^2(B_{2r})}.
\end{equation}
 Similarly, as $\|\varepsilon_{ij}\|_{L^\infty}\lesssim \|m\|_{\mathcal C}^{1/2}$ we have for $II_{41}$
\begin{equation}\label{Eqqq-29a}
II_{41}\leq C(n)\|m\|_{\mathcal C}^{1/2}\|S_{2,b}(u)\|^2_{L^p(B_{2r})}.
\end{equation}

Secondly, the Carleson condition \eqref{CarbarA} and the Cauchy-Schwarz inequality imply
\begin{equation}\label{Eqqq-30}\nonumber
II_{1} + II_{2}+II_{42}+J_1 \leq C(n)(1+\|\mu\|_{\mathcal C}^{1/2}+\|m\|_{\mathcal C}^{1/2})
\|S_{2,b}(u)\|_{L^2(B_{2r})}\|\tilde{N}_{2,a}(u)\|_{L^2(B_{2r})}.
\end{equation}
Next, corresponding to the case when the derivative falls on the cutoff function $\zeta$ we have
\begin{align}\label{TDWW}
J_{3}+II_{43}+III &\leq C(\Lambda,n)(1+\|m\|_{\mathcal C}^{1/2})\iint_{[0,2r]\times B_{2r}}\left|\nabla v\right||v| \frac{x_0}{r}\,dx_0\,dx' 
\nonumber\\[4pt]
&\leq C\left(\iint_{[0,2r]\times B_{2r}}|v|^{2}\frac{x_0}{r^{2}}\,dx_0\,dx'\right)^{1/2} 
\|S^{2r}_{2,b}(v)\|_{L^2(B_{2r})} 
\nonumber\\[4pt]
&\leq C\|S_{2,b}(u)\|_{L^2(B_{2r})}\|\tilde{N}_{2,a}(u)\|_{L^2(B_{2r})}.
\end{align}
Finally, the interior term $IV$, which arises from the fact that $\partial_{0}\zeta$ vanishes on the set
$(0,r)\cup(2r,\infty)$ may be estimated as follows:
\begin{equation}\label{Eqqq-31}
IV\leq\frac{c}{r}\iint_{[r,2r]\times B_{2r}}|v|^{2}\,dx_0\,dx'.
\end{equation}
Summing up all terms, the above analysis ultimately yields
\begin{align}\label{E1:uonh}
&\hskip -0.20in \int_{B_{r}(y')}|v(0,x')|^2\,dx' 
\nonumber\\[4pt]
&\hskip 0.40in 
\leq C(\Lambda,n)(1+\|\mu\|^{1/2}_{\mathcal C}+\|m\|^{1/2}_{\mathcal C}) 
\|S_{2,b}(u)\|_{L^2(B_{2r})}\|\tilde{N}_a(u)\|_{L^2(B_{2r})}
\nonumber\\[4pt]
&\hskip 0.40in 
\quad+C(\Lambda,p,n)(1+\|m\|^{1/2}_{\mathcal C})\|S_{2,b}(u)\|^2_{L^2(B_{2r})}
+\frac{c}{r}\iint_{[r,2r]\times B_{2r}}|v|^2\,dx_0\,dx'.
\end{align}
With this in hand, the estimate in \eqref{TTBBMM} follows (by passing from $v$ back to $u$ via the map $\rho$).

The case $r>>h$ requires some extra care. However we can observe that for $\theta\phi(x')\ge h$ we have $u(\theta\phi(x'),x')=0$ and hence for such points the lefthand side of \eqref{TTBBMM} vanishes. It follows that without loss of generality we may modify our function $\phi$ assume that $\theta\phi\le h$ in $\Delta_r$ without changing the value of the lefthand side of  \eqref{TTBBMM}. What this implies is that the estimate   \eqref{TTBBMM} for $\Delta_r$ can be deduced from adding up estimates  \eqref{TTBBMM} for smaller balls $\Delta_{r'}\subset\Delta_r$ where $r'\approx h$ and hence we still have $\phi\le 2r'$. However, the estimate for such small balls was established above and hence we can conclude that \eqref{TTBBMM} holds for balls of all sizes.

Finally, the fact that \eqref{Eqqq-25} can replace the integral over $\mathcal K$ in \eqref{TTBBMM} is a consequence of the Poincar\'e inequality.
\end{proof}

From inequality \eqref{TTBBMM} of Lemma \ref{S3:L8}, one can derive the global, and local, domination of the nontangential maximal function by the square
function, $S_2$,  in the $L^2$ norm exactly as in Proposition 5.8 of \cite{DHM}. The passage from $S_2$ to the $p$-adapted square function, $S_p$ is carried out in 
\cite{DPcplx}.

\section{Proofs of the main results.}

We briefly discuss the proof of Theorem \ref{MainPerturb}, since the calculations of (4.8-4.10) in \cite{DPcplx}, and the arguments that follow,
carry over verbatim. The only difference is the presence of another constant $\|m\|^{1/2}_{\mathcal C}$ which must be sufficiently small. The argument
is carried out for solutions in the infinite strip of height $h$, which implies finiteness of certain square functions and nontangential maximal functions.
We also do need to assume that the coefficients of the equations are smooth, again in order to have finiteness of the appropriate quantities (to apply
Proposition \ref{S3:C7}); an approximation and limiting argument removes this assumption in the end. The cases $p\geq 2$ and $p<2$ are argued separately, 
primarily because the finiteness assumptions are established differently.

\medskip

Next we establish Theorem \ref{Tmain}.
\begin{proof} 
With the perturbation result in hand, we may now introduce a mollification of the coefficients, as in  \cite{DPP}. Let ${\mathcal L}_1:={\mathcal L}$
be an operator whose coefficients satisfy assumptions of Theorem \ref{Tmain}. Denote these coefficients by $A^1$, $B^1$ so that we have than
\begin{equation}\label{oscT1a}
d\mu_1 =  \left( \delta(x)^{-1}\left(\mbox{osc}_{B_{\delta(x)/2}(x)}A^1\right)^2 + \sup_{B_{\delta(x)/2}(x)} |B^1|^2 \delta(x)\right) dx
\end{equation}
is a Carleson measure with norm $\|\mu_1\|_{\mathcal C}$ and $A^1$ is $p$-elliptic.

Consider a new operator ${\mathcal L}_0$ whose coefficients are
defined as follows. 
Set
$$A^0(x_0,x')= \iint_{{\mathbb R}^n_+} A^1(s,u)\phi_{t}(s-t,x'-u)dsdu,$$
where $\phi$ is a smooth real, nonnegative bump function on ${\mathbb R}^n$ supported in
the ball  $B_{1/2}(0)$ such that $\iint \phi=1$ and
$\phi_t(s,y) = t^{-n}\phi(s/t,y/t)$. We also set $B^0=B^1$.

The Carleson norms of
\begin{equation}\label{Car_hatAA-X}
d{\mu_0}(x)=\sup_{B_{\delta(x)/2}(x)}\left[|\nabla{A^0}|^{2} + |B^0|^{2} \right]\delta(x)\,dx
\end{equation}
and 
\begin{equation}\label{eqdm-X}
 dm(x)=\sup_{B_{\delta(x)/2}(x)}\left[|A^0 - A^1|^{2} \delta^{-1}(x)+ |B^0 - B^1|^{2} \delta(x)\right]\,dx
\end{equation}
satisfy, by the same arguments of  \cite[Corollary 2.3]{DPP}, the following bounds:
$$\|\mu_0\|_{\mathcal C}+\|m\|_{\mathcal C}\le C\|\mu_1\|_{\mathcal C},$$
for some $C=C(n,\phi)\ge 1$.

In the complex coefficient setting, there are a couple of new points to check. First, $p$-ellipticity is preserved by the mollification, hence if ${\mathcal L}_1$ is $p$-elliptic then so is ${\mathcal L}_0$. Second, the conditions $A^0_{00}=1$ and $\mathscr{I}m\,A^0_{0j}=0$ for all $1\leq j \leq n-1$ are preserved as well, provided the operator ${\mathcal L}_1$ is in canonical form. 

It follows that Proposition \ref{S3:C7} applies to ${\mathcal L}_0$ and ${\mathcal L}_1$ and we have for any energy solution $\mathcal L_1u=0$ the inequality

\begin{equation}\label{dferg}
\int_{\BBR^{n-1}}\left[\tilde{N}_{p,a}(u)\right]^{p}\,dx'\le C\int_{\BBR^{n-1}}\left[{S}_{p,a}(u)\right]^{p}\,dx'.
\end{equation}
Observe that in Lemma \ref{S3:L4} we only require that $\mu'$ is a Carleson measure. Hence to apply this lemma we do not have to mollify all coefficients of ${\mathcal L}_1$, only those in the row $A_{0j}^1$.
Define a second operator $\bar{\mathcal L}_0$ whose coefficients are:
\begin{align}\nonumber
\bar{A}^0_{ij}(x_0,x')&=\begin{cases}{A}^1_{ij}(x_0,x'),&\quad\mbox{for }i>0,\\
 \iint_{{\mathbb R}^n_+} A^1_{ij}(s,u)\phi_{t}(s-t,x'-u)dsdu,&\quad\mbox{for }i=0,\end{cases}\\
 \bar{B}^0_{i}(x_0,x')&={B}^1_{i}(x_0,x').
\end{align}
Clearly, if
\begin{equation}\label{eqdm-X1}
 d\bar{m}(x)=\sup_{B_{\delta(x)/2}(x)}\left[|\bar{A}^0 - A^1|^{2} \delta^{-1}(x)+ |\bar{B}^0 - B^1|^{2} \delta(x)\right]\,dx
\end{equation}
and
\begin{equation}\label{eqdm-X2}
 d\bar{\mu}'(x)=\sup_{B_{\delta(x)/2}(x)}\left[\textstyle\sum_j\left|\partial_0 \bar{A}^0_{0j}\right|^{2}+\left|\textstyle\sum_j\partial_j \bar{A}^0_{0j}\right|^{2} + |\bar{B}^0|^{2} \right]\delta(x)\,dx.
\end{equation}
then $\|\bar{m}\|_{\mathcal C},|\bar{\mu}'\|_{\mathcal C}\lesssim \|\mu_1'\|_{\mathcal C}$, where
\begin{equation}\label{oscT1ax}
d\mu'_1 =  \left(\delta(x)^{-1}\sum_{j=0}^{n-1}\left(\mbox{osc}_{B_{\delta(x)/2}(x)}A_{0j}^1\right)^2+ \sup_{B_{\delta(x)/2}(x)} |B^1|^2 \delta(x)\right)  dx
\end{equation}
It follows that Lemma \ref{S3:L4} applies to $\bar{\mathcal L}_0$ and ${\mathcal L}_1$ and gives us
\begin{eqnarray}\nonumber
\int_{\BBR^{n-1}}\left[{S}_{p,a}(u)\right]^{p}\,dx'
&\leq& C_1\int_{\BBR^{n-1}}|u(0,x')|^{p}\,dx'\\
&+&C_2\|\mu_1'\|^{1/2}_{\mathcal{C}}\int_{\BBR^{n-1}}\left[\tilde{N}_{p,a}(u)\right]^{p}\,dx'.
\end{eqnarray}
This combined with \eqref{dferg} implies the estimate \eqref{Main-Est2} and hence solvability of the $L^p$ Dirichlet problem for the operator ${\mathcal L}={\mathcal L}_1$, provided $\|\mu'_1\|_{\mathcal C}$ is sufficiently small. 

In particular, when the operator ${\mathcal L}$ has the block form (i.e. $A_{0j}^1=A^1_{j0}=\delta_{0j}$ and $B^1_j=0$) then $\mu_1'=0$ and hence Corollary \ref{block} holds.
\end{proof}

\end{document}